\documentclass[11pt]{amsart}

\usepackage[T1]{fontenc}
\usepackage{lmodern}
\usepackage{amsmath,amssymb,amsfonts,amsthm,mathtools}
\usepackage{booktabs,array}
\usepackage{enumitem}
\usepackage{microtype}
\usepackage[margin=1.05in]{geometry}
\usepackage[colorlinks=true,linkcolor=blue,citecolor=blue,urlcolor=blue]{hyperref}
\usepackage{bookmark}

\hypersetup{
  pdftitle={Positive-regularity norm inflation for the 3D hypodissipative Navier--Stokes equations},
  pdfauthor={Guirong Tang and Shiyang Xiong},
  pdfsubject={Norm inflation for hypodissipative Navier--Stokes equations},
  pdfkeywords={hypodissipative Navier--Stokes equations, norm inflation, strong ill-posedness, vortex-ring mixing}
}

\numberwithin{equation}{section}
\allowdisplaybreaks
\newtheorem{theorem}{Theorem}[section]
\newtheorem{proposition}[theorem]{Proposition}
\newtheorem{lemma}[theorem]{Lemma}
\newtheorem{corollary}[theorem]{Corollary}
\theoremstyle{definition}

\theoremstyle{remark}
\newtheorem{remark}[theorem]{Remark}

\newcommand{\R}{\mathbb R}

\newcommand{\veps}{\varepsilon}
\newcommand{\supp}{\operatorname{supp}}

\newcommand{\dd}{\,\mathrm d}
\newcommand{\divg}{\operatorname{div}}
\newcommand{\curl}{\operatorname{curl}}
\newcommand{\PP}{\mathbb P}
\newcommand{\ubar}{\bar u}
\newcommand{\Lambdaop}{\Lambda}
\newcommand{\norm}[2]{\left\|#2\right\|_{#1}}
\newcommand{\abs}[1]{\left|#1\right|}
\newcommand{\ip}[2]{\left\langle #1,#2\right\rangle}
\newcommand{\lesssimz}{\lesssim_{\zeta}}
\title[Norm inflation for hypodissipative Navier--Stokes]{Positive-regularity norm inflation for the 3D hypodissipative Navier--Stokes equations}

\author{Guirong Tang}
\address{Guirong Tang\\School of Mathematical Sciences, Capital Normal University, Beijing 100048, China}
\email{tangguirong@amss.ac.cn}

\author{Shiyang Xiong}
\address{Shiyang Xiong\\
Institute of Mathematics, Academy of Mathematics and Systems Science\\
Chinese Academy of Sciences, Beijing 100190, China}
\email{xiongshiyang@amss.ac.cn}

\date{}


\begin{document}

\begin{abstract}
We prove same-space norm inflation for the three-dimensional hypodissipative Navier--Stokes equations with dissipation $(-\Delta)^\alpha$, $0<\alpha<1$.  Let $2\le p<\infty$ and
\[
  0<s<1-2\alpha+\frac3p.
\]
In the Besov case, let also $1\le q\le\infty$.  There exist divergence-free $C_c^\infty(\mathbb R^3)$ initial data that are arbitrarily small in $W^{s,p}$, respectively in $B^s_{p,q}$, while the corresponding unique local smooth solution becomes arbitrarily large in the same space in arbitrarily short time.  The proof adapts the anisotropic vortex-ring mixing mechanism to fractional dissipation; the strict scaling-supercritical gap makes both the curvature error and the nonlocal dissipative error perturbative.
\end{abstract}

\maketitle

\section{Introduction}\label{sec:intro}

We consider the three-dimensional hypodissipative Navier--Stokes equations
\begin{equation}\label{eq:fns}
\begin{cases}
  \partial_tu+u\cdot\nabla u+\nabla\pi+(-\Delta)^\alpha u=0,\\
  \divg u=0,\\
  u|_{t=0}=u_0,
\end{cases}
\qquad (t,x)\in[0,T]\times\R^3,
\end{equation}
where $0<\alpha<1$.  The scaling
\begin{equation}\label{eq:scaling}
  u(t,x)\longmapsto u_\lambda(t,x)
  :=\lambda^{2\alpha-1}u(\lambda^{2\alpha}t,\lambda x)
\end{equation}
gives
\begin{equation}\label{eq:scaling-norm}
  \norm{\dot W^{s,p}}{u_{0,\lambda}}
  =\lambda^{s+2\alpha-1-3/p}\norm{\dot W^{s,p}}{u_0}.
\end{equation}
Thus the scaling-critical index is
\begin{equation}\label{eq:critical-index}
  s_c(p,\alpha):=1-2\alpha+\frac3p.
\end{equation}
The range $0<s<s_c(p,\alpha)$ is positive-regularity and strictly scaling-supercritical.

For the classical equations, Luo \cite{Luo2024,Luo2025} showed that a thin vortex ring can generate norm inflation through physical-space mixing.  This raises a natural question for \eqref{eq:fns}: does the same mechanism survive the nonlocal dissipation $(-\Delta)^\alpha$, and does it reach the full positive strictly supercritical range?  The following theorem gives an affirmative answer.

\begin{theorem}[Norm inflation]\label{thm:main}
Let
\[
  0<\alpha<1,
  \qquad 2\le p<\infty,
  \qquad 0<s<1-2\alpha+\frac3p.
\]
Let $X=W^{s,p}$, or let $X=B^s_{p,q}$ for a fixed $1\le q\le\infty$.  For every $\veps>0$, there exist a divergence-free vector field $u_0\in C_c^\infty(\R^3)$ and a time $0<t_\veps\le\veps$ such that the unique smooth solution of \eqref{eq:fns} exists on $[0,t_\veps]$ and satisfies
\[
  \norm{X}{u_0}\le\veps,
  \qquad
  \norm{X}{u(t_\veps)}\ge\veps^{-1}.
\]
\end{theorem}

\begin{remark}[Interpretation and scope]\label{rem:scope}
Throughout the paper, $W^{s,p}$ denotes the Bessel-potential Sobolev space defined by $J^su\in L^p$, where $J=(1-\Delta)^{1/2}$.  Theorem \ref{thm:main} is a same-space statement: the initial and inflated norms have the same regularity and, in the Besov case, the same fixed summability index $q$.  The data are smooth and compactly supported, and the inflation occurs along the unique local smooth solution; in particular, the smooth solution map cannot be locally bounded at the origin in any of the stated spaces.

The upper boundary $s=s_c(p,\alpha)$ is the scaling line.  The proof uses the strict gap
\[
  \delta:=1-2\alpha+\frac3p-s>0
\]
twice: to separate the inflation time from the dissipative time scale and to make the full-space fractional residual perturbative.  At $\delta=0$ the present parameter balance loses both gains.  This identifies a limitation of the construction, not a conclusion about endpoint well-posedness or ill-posedness.  The same geometric construction formally includes $\alpha=1$, where the corresponding positive-regularity Sobolev and Besov results are already contained in \cite{Luo2024,Luo2025}.
\end{remark}

\subsection{Background and related work}
Norm inflation for the classical Navier--Stokes equations was first developed in negative-regularity spaces through Fourier-space high-to-low cascades; see, among others, \cite{BourgainPavlovic2008,Yoneda2010,Wang2015}.  Cheskidov and Dai \cite{CheskidovDai2014} extended this type of construction to generalized Navier--Stokes equations with dissipation powers $\alpha\ge1$.  These arguments transfer energy between separated Fourier scales and are distinct from the physical-space mechanism used below.

The positive-regularity theory has a different geometry.  Luo \cite{Luo2024} introduced an anisotropic vortex-ring construction for the Euler and classical Navier--Stokes equations: a nearly two-dimensional meridional flow transports a swirl and creates rapid oscillations across a thin torus.  The later work \cite{Luo2025} developed the corresponding Besov picture and separated the positive-regularity mixing regime from the negative-regularity un-mixing regime.  The present paper follows the positive branch of that construction.  The new issue is not the production of oscillations itself, but the stability of the mechanism under the nonlocal operator $(-\Delta)^\alpha$ throughout the range $0<\alpha<1$.

For fractional Navier--Stokes equations, critical-space well-posedness and weak- or mild-solution nonuniqueness have been studied from several perspectives; representative references include \cite{MiaoYuanZhang2008,Zhai2010,YuZhai2012,ColomboDeLellisDeRosa2018,DeRosa2019,AlbrittonColombo2023}.  Those results concern different regularity regimes or different solution classes.  Here the data are $C_c^\infty$, the solution is the unique smooth one, and the growth is measured in the same positive supercritical space as the initial datum.

The vortex-ring geometry and the transport mechanism are adapted from \cite{Luo2024,Luo2025}.  The principal new point is the treatment of the fractional dissipative error on the whole space.  We show that the same strict supercritical gap simultaneously makes the cumulative dissipation negligible on the inflation time and places $(-\Delta)^\alpha\ubar$ below the curvature error.  This is the mechanism that extends positive-regularity norm inflation from the classical equation to the full hypodissipative range.

\subsection{Outline of the construction}
The proof has three steps.

\smallskip
\noindent\textit{Step 1: fractional Navier--Stokes versus Euler dynamics.}
We construct a compactly supported approximate solution $\ubar$ on a thin torus of transverse thickness $\mu^{-1}$ and radius $\nu^{-1}$, with $\mu\gg\nu$.  The geometric error gains the aspect ratio $\nu/\mu$.  The nonlocal term $(-\Delta)^\alpha\ubar$ is estimated globally by Fourier multipliers.  With a suitable amplitude and inflation time, the strict gap $\delta>0$ yields
\[
  t_*\mu^{2\alpha}\longrightarrow0,
  \qquad
  \frac{\mu^{2\alpha}}{A\nu}\longrightarrow0
  \qquad (\mu\to\infty),
\]
after the structural parameter is fixed.

\smallskip
\noindent\textit{Step 2: three-dimensional Euler versus steady two-dimensional transport.}
In toroidal coordinates, the meridional velocity is generated by a two-dimensional steady Euler flow, while the swirl solves a linear transport equation.  Its phase develops a frequency $K_*$ that is a prescribed large multiple of the original transverse frequency.  A weighted non-cancellation argument gives integer derivative growth.  A one-scale finite-difference estimate yields the Besov lower bound, and homogeneous interpolation yields the Sobolev lower bound.

\smallskip
\noindent\textit{Step 3: stability of the approximation.}
Let $w=u-\ubar$.  Under a Lipschitz bootstrap, finite-exponent Sobolev norms of the exact solution propagate with a coefficient depending only on $\|\nabla u\|_{L^\infty}$.  These estimates provide the endpoint bounds needed in an anisotropic $H^k$ induction for $w$.  Interpolation then shows that the target-space error is negligible compared with the growth of $\ubar$.

\subsection{Organization of the paper}
Section \ref{sec:prelim} records the function-space estimates and the smooth local theory.  Section \ref{sec:approx} constructs the approximate solution and proves its residual bounds.  Section \ref{sec:growth} establishes norm growth of the transported swirl.  Section \ref{sec:stability} proves the projected commutator estimate and the stability of the approximation up to the inflation time.  Section \ref{sec:main-proof} completes the parameter choice and proves Theorem \ref{thm:main}.

\section{Preliminaries}\label{sec:prelim}

\subsection{Function spaces}
Set $\Lambdaop=(-\Delta)^{1/2}$ and $J=(1-\Delta)^{1/2}$.  We use the Bessel-potential Sobolev spaces $W^{\sigma,p}$, the Besov spaces $B^\sigma_{p,q}$, and their homogeneous counterparts $\dot W^{\sigma,p}$ and $\dot B^\sigma_{p,q}$.  All spaces are over $\R^3$ unless otherwise stated.  Standard facts used below may be found in \cite{BahouriCheminDanchin2011,BerghLofstrom1976,Triebel1983}.

For $1<p<\infty$ and $\sigma\ge0$,
\[
  \norm{W^{\sigma,p}}{f}\simeq\norm{L^p}{J^\sigma f},
  \qquad
  \norm{\dot W^{\sigma,p}}{f}\simeq\norm{L^p}{\Lambdaop^\sigma f}.
\]
For $s>0$ and an integer $m>s$, a directional finite-difference characterization gives
\begin{equation}\label{eq:directional-besov}
  \norm{\dot B^s_{p,q}}{f}
  \gtrsim
  \left(\int_0^\infty
  \bigl[h^{-s}\norm{L^p}{\Delta_{he_3}^mf}\bigr]^q\frac{\dd h}{h}\right)^{1/q},
\end{equation}
with the usual supremum when $q=\infty$.  Here $\Delta_hf(x)=f(x+h)-f(x)$.  We also use the homogeneous Sobolev interpolation inequality
\begin{equation}\label{eq:sobolev-interpolation-prelim}
  \norm{\dot W^{\sigma_1,p}}{f}
  \lesssim
  \norm{\dot W^{\sigma_0,p}}{f}^{1-\theta}
  \norm{\dot W^{\sigma_2,p}}{f}^{\theta},
  \qquad
  \sigma_1=(1-\theta)\sigma_0+\theta\sigma_2.
\end{equation}
For $2\le p<\infty$, let
\begin{equation}\label{eq:dp}
  d_p:=3\left(\frac12-\frac1p\right).
\end{equation}
Then
\begin{equation}\label{eq:sobolev-transfer-prelim}
  \norm{\dot W^{s,p}}{f}
  \lesssim\norm{\dot H^{s+d_p}}{f}.
\end{equation}
Moreover, for every $\eta>0$ and $1\le q\le\infty$,
\begin{equation}\label{eq:besov-transfer-prelim}
  \norm{B^s_{p,q}}{f}
  \lesssim_{s,p,q,\eta}\norm{H^{s+d_p+\eta}}{f}.
\end{equation}
The arbitrarily small loss $\eta$ in \eqref{eq:besov-transfer-prelim} permits all values of $q$.

\subsection{Smooth local theory}

We use the following standard local theory and continuation criterion.

\begin{proposition}[Smooth local theory]\label{prop:local-theory}
Let $0<\alpha\le1$ and let $m>5/2$ be an integer.  For every divergence-free $u_0\in H^m(\R^3)$, there exists $T>0$ and a unique solution
\[
  u\in C([0,T];H^m)\cap L^2(0,T;H^{m+\alpha})
\]
of \eqref{eq:fns}.  If $u_0\in C_c^\infty$, the solution remains smooth as long as
\begin{equation}\label{eq:continuation-criterion}
  \int_0^T\norm{L^\infty}{\nabla u(t)}\dd t<\infty.
\end{equation}
\end{proposition}

\begin{proof}
Apply a smooth Friedrichs cutoff $J_M$ to the projected equation and set $v_M=J_Mu_M$.  Choosing $J_M$ self-adjoint, the nonlinear energy term reduces to
\[
  \ip{D^\beta J_M\PP(v_M\cdot\nabla v_M)}{D^\beta u_M}
  =\ip{D^\beta(v_M\cdot\nabla v_M)}{D^\beta v_M}.
\]
The divergence-free cancellation and the standard Moser estimate give, uniformly in $M$,
\[
  \frac12\frac{\dd}{\dd t}\norm{H^m}{u_M}^2
  +\norm{H^m}{\Lambdaop^\alpha u_M}^2
  \le C_m\norm{L^\infty}{\nabla u_M}\norm{H^m}{u_M}^2.
\]
Since $m>5/2$, this yields an $M$-independent local interval.  Compactness gives existence, and the $L^2$ difference estimate gives uniqueness.  The same inequality, followed by the usual higher-order bootstrap, proves the continuation criterion \eqref{eq:continuation-criterion}.  This is the standard Friedrichs construction; see also \cite{BahouriCheminDanchin2011}.
\end{proof}

\section{The approximate solution}\label{sec:approx}

\subsection{Parameters}
Fix $0<\alpha<1$, $2\le p<\infty$, and
\[
  0<s<1-2\alpha+\frac3p.
\]
Set
\begin{equation}\label{eq:delta-b}
  \delta:=1-2\alpha+\frac3p-s>0,
  \qquad
  b:=\frac{\delta}{100},
\end{equation}
and choose an integer $N$ such that
\begin{equation}\label{eq:N-choice}
  N\ge100+\frac{10}{s}.
\end{equation}
Throughout Sections \ref{sec:approx}--\ref{sec:stability}, the structural parameter $0<\zeta\le1$ is fixed first and the large parameter $\mu$ is then taken sufficiently large depending on $\zeta$.  Constants denoted by $C_\zeta$ may depend on the fixed exponents, profiles, and $\zeta$, but not on $\mu$ or $\nu$.

Set
\begin{equation}\label{eq:nu-theta}
  \nu:=\mu^{1-b},
  \qquad
  \vartheta:=\frac\nu\mu=\mu^{-b},
\end{equation}
and define
\begin{equation}\label{eq:amplitude-scales}
  A:=\zeta^2\mu^{-s+2/p}\nu^{1/p},
  \qquad
  V_j:=\mu^{-2/j}\nu^{-1/j},
  \qquad
  S_{\sigma,j}:=A\mu^\sigma V_j,
  \qquad
  L:=A\mu.
\end{equation}
For $j=\infty$, we use $V_\infty=1$.  The target scale is
\begin{equation}\label{eq:target-scale}
  S_{s,p}=A\mu^sV_p=\zeta^2.
\end{equation}
Finally, let
\begin{equation}\label{eq:tstar-Kstar}
  t_*:=\zeta^{-N}L^{-1},
  \qquad
  K_*:=t_*A\mu^2=\zeta^{-N}\mu.
\end{equation}
Thus $K_*/\mu=\zeta^{-N}$ is a prescribed large frequency separation once $\zeta$ is chosen.

\subsection{Toroidal coordinates and profile estimates}
Use cylindrical coordinates $(r,\theta,z)$ in $\R^3$.  For positive parameters $\mu\gg\nu$, define
\begin{equation}\label{eq:toroidal-coordinates}
  \rho:=\sqrt{(r-\nu^{-1})^2+z^2},
  \qquad
  r-\nu^{-1}=\rho\cos\varphi,
  \qquad
  z=\rho\sin\varphi.
\end{equation}
On the region $\rho\simeq\mu^{-1}$ one has
\begin{equation}\label{eq:r-comparison}
  r\simeq\nu^{-1},
  \qquad
  \abs{\nabla^jr^{-1}}\lesssim_j\nu^{j+1},
  \qquad
  \abs{\nabla^j\rho}\lesssim_j\mu^{j-1},
  \qquad
  \abs{\nabla^j\varphi}\lesssim_j\mu^j.
\end{equation}
The volume of a fixed-width toroidal profile supported there is comparable to $\mu^{-2}\nu^{-1}$, so the $L^\ell$ volume factor is precisely $V_\ell$ from \eqref{eq:amplitude-scales}.

\begin{lemma}[Toroidal profile estimates]\label{lem:toroidal-calculus}
Let $H\in C_c^\infty((0,\infty)\times\mathbb S^1)$ be supported where its first variable belongs to a fixed compact subset of $(0,\infty)$.  Let $e$ be any one of the cylindrical unit vectors $e_r,e_\theta,e_z$, and let
\[
  F(x)=aH(\mu\rho,\varphi)e.
\]
If $\mu/\nu$ is sufficiently large, then, for every integer $k\ge0$ and $1\le\ell\le\infty$,
\begin{equation}\label{eq:toroidal-integer}
  \norm{L^\ell}{\nabla^kF}
  \le C_{k,\ell,H}\abs{a}\mu^kV_\ell.
\end{equation}
For every real $\sigma\ge0$ and $1<\ell<\infty$,
\begin{equation}\label{eq:toroidal-fractional}
  \norm{L^\ell}{\Lambdaop^\sigma F}
  \le C_{\sigma,\ell,H}\abs{a}\mu^\sigma V_\ell.
\end{equation}
The same bounds hold for finite sums of profiles containing factors $(\mu r)^{-j}$, with the additional gain $(\nu/\mu)^j$.
\end{lemma}

\begin{proof}
In the support of $H(\mu\rho,\varphi)$, \eqref{eq:r-comparison} holds and the cylindrical basis derivatives cost powers of $r^{-1}\simeq\nu$, which are no larger than the transverse scale $\mu$.  Repeated use of the chain rule therefore gives the pointwise bound $\abs{\nabla^kF}\lesssim\abs a\mu^k$.  Integration over a set of volume comparable to $\mu^{-2}\nu^{-1}$ yields \eqref{eq:toroidal-integer}.  The fractional estimate follows by interpolation between two neighboring integer orders and the boundedness of Riesz transforms on $L^\ell$; see, for example, \cite{Stein1970}.  Every factor $(\mu r)^{-1}$ is $O(\nu/\mu)$, proving the final assertion.
\end{proof}

\begin{lemma}[Derivatives of the transported phase]\label{lem:phase-upper-prelim}
Let
\begin{equation}\label{eq:phase-def-prelim}
  \Phi(t,\rho,\varphi):=\varphi-\frac{tA}{\rho},
  \qquad
  \mathcal K(t):=\mu+tA\mu^2.
\end{equation}
On $\rho\simeq\mu^{-1}$, for every integer $j\ge1$,
\begin{equation}\label{eq:phase-upper-prelim-est}
  \abs{\nabla^j\Phi(t)}
  \lesssim_j \mathcal K(t)\mu^{j-1}.
\end{equation}
Consequently, if $0\le t\le t_*$ and $t_*A\mu^2=\zeta^{-N}\mu$, then $\mathcal K(t)\le(1+\zeta^{-N})\mu$ and all profile estimates remain valid with constants depending on the fixed $\zeta$ but not on $\mu$ or $\nu$.
\end{lemma}

\begin{proof}
The estimate follows from \eqref{eq:r-comparison} and repeated differentiation of $\rho^{-1}$.  The last statement is immediate from the definition of $t_*$.
\end{proof}

\subsection{The vortex-ring ansatz}
Choose $f,g\in C_c^\infty(\R)$ such that
\begin{equation}\label{eq:profile-choice}
\begin{gathered}
  f'=1\quad\text{on a neighborhood of }[1,3/2],
  \qquad \supp f\subset(1/2,2),\\
  0\le g\in C_c^\infty((1,3/2)),
  \qquad g\not\equiv0.
\end{gathered}
\end{equation}
Fix an interval
\begin{equation}\label{eq:Ig}
  I_g\Subset(1,3/2)
  \quad\text{such that}\quad
  g\ge c_g>0\quad\text{on }I_g.
\end{equation}
In the toroidal coordinates \eqref{eq:toroidal-coordinates}, define
\begin{equation}\label{eq:initial-components}
\begin{aligned}
  u_{0,\theta}&:=Ag(\mu\rho)\sin\varphi,\\
  u_{0,r}&:=-Af'(\mu\rho)\partial_z\rho,\\
  u_{0,z}&:=Af'(\mu\rho)\partial_r\rho,
\end{aligned}
\qquad
  u_c:=A\mu^{-1}\frac{f(\mu\rho)}r,
\end{equation}
and set
\begin{equation}\label{eq:u0}
  u_0:=u_{0,\theta}e_\theta+u_{0,r}e_r+(u_{0,z}+u_c)e_z.
\end{equation}
For $\mu/\nu$ large, the support is contained in a torus separated from the symmetry axis.  Hence $u_0\in C_c^\infty(\R^3)$.

\subsection{The transported swirl and the approximate equation}
Define
\begin{equation}\label{eq:approx-components}
  \ubar_\theta(t):=Ag(\mu\rho)\sin\Phi(t,\rho,\varphi),
  \qquad
  \Phi(t,\rho,\varphi):=\varphi-\frac{tA}{\rho},
\end{equation}
and
\begin{equation}\label{eq:approx-field}
  \ubar(t):=u_{0,r}e_r+\ubar_\theta(t)e_\theta+(u_{0,z}+u_c)e_z.
\end{equation}
Then $\ubar(0)=u_0$ and $\divg\ubar(t)=0$ for all $t$.

On the support of $g(\mu\rho)$, one has $f'(\mu\rho)=1$.  Since
\begin{equation}\label{eq:meridional-transport}
  u_{0,r}\partial_r+u_{0,z}\partial_z
  =\frac{A}{\rho}\partial_\varphi,
\end{equation}
the swirl satisfies
\begin{equation}\label{eq:swirl-transport}
  \partial_t\ubar_\theta
  +(u_{0,r}\partial_r+u_{0,z}\partial_z)\ubar_\theta=0.
\end{equation}

We next isolate the pressure generated by the meridional field.  Let
\[
  v:=(u_{0,r},u_{0,z})=\nabla_{r,z}^\perp\psi,
  \qquad
  \psi=A\mu^{-1}f(\mu\rho).
\]
Its scalar vorticity $\omega=\partial_ru_{0,z}-\partial_zu_{0,r}$ is a function of $\rho$ alone.  Hence $v\cdot\nabla_{r,z}\omega=0$ and
\begin{equation}\label{eq:curl-meridional-convection}
  \curl_{r,z}\bigl((v\cdot\nabla_{r,z})v\bigr)
  =v\cdot\nabla_{r,z}\omega=0.
\end{equation}
Because $\R^2$ is simply connected, there exists a smooth scalar $\pi_{\mathrm E}(r,z)$ such that
\begin{equation}\label{eq:meridional-euler}
  (v\cdot\nabla_{r,z})v+\nabla_{r,z}\pi_{\mathrm E}=0.
\end{equation}
The convective field $(v\cdot\nabla_{r,z})v=-\nabla_{r,z}\pi_{\mathrm E}$ in \eqref{eq:meridional-euler} is compactly supported, so $\nabla\pi_{\mathrm E}$ is compactly supported and $\pi_{\mathrm E}$ is constant on each component of its complement.  The torus is separated from $r=0$, and $\pi_{\mathrm E}$ is constant near the axis.  Its axisymmetric lift is therefore smooth on $\R^3$.

Using the cylindrical form of the nonlinear term and \eqref{eq:swirl-transport}--\eqref{eq:meridional-euler}, one obtains
\begin{equation}\label{eq:approx-equation}
  \partial_t\ubar+\ubar\cdot\nabla\ubar+\nabla\pi_{\mathrm E}
  +(-\Delta)^\alpha\ubar=E,
  \qquad
  \divg\ubar=0,
\end{equation}
where
\begin{equation}\label{eq:E-decomposition}
  E=E_{\mathrm{geom}}+E_{\mathrm{diss}},
  \qquad
  E_{\mathrm{diss}}:=(-\Delta)^\alpha\ubar,
\end{equation}
and the cylindrical components of the geometric residual are
\begin{equation}\label{eq:geometric-residual-components}
\begin{aligned}
  (E_{\mathrm{geom}})_\theta
  &:=u_c\partial_z\ubar_\theta+\frac{u_{0,r}}r\ubar_\theta,\\
  (E_{\mathrm{geom}})_r
  &:=u_c\partial_zu_{0,r}-\frac{\ubar_\theta^2}{r},\\
  (E_{\mathrm{geom}})_z
  &:=u_{0,r}\partial_ru_c+u_{0,z}\partial_zu_c
    +u_c\partial_zu_{0,z}+u_c\partial_zu_c.
\end{aligned}
\end{equation}
Thus the centrifugal term, the cylindrical swirl term, and all cross terms involving $u_c$ are explicitly retained.

\begin{proposition}[Approximate solution estimates]\label{prop:approximation-package}
Fix a target space
\[
  X=W^{s,p}
  \qquad\text{or}\qquad
  X=B^s_{p,q}\quad\text{with $q$ fixed}.
\]
For every fixed $\sigma\ge0$, integer $k\ge0$, and $1<j<\infty$, there exists $\mu_0(\zeta,\sigma,k,j)$ such that, for $\mu\ge\mu_0$, the following statements hold on $[0,t_*]$:
\begin{enumerate}[label=\textup{(\roman*)},leftmargin=2.2em,itemsep=2pt]
\item $u_0,\ubar(t)\in C_c^\infty(\R^3)$ are divergence free and $\ubar(0)=u_0$;
\item for a constant $C_{\mathrm{in}}(X)>0$ independent of $\zeta$ and $\mu$,
\begin{equation}\label{eq:initial-target-smallness}
  \norm{X}{u_0}\le C_{\mathrm{in}}(X)\zeta^2;
\end{equation}
\item the profile estimate
\begin{equation}\label{eq:approx-profile-scale}
  \norm{W^{\sigma,j}}{\ubar(t)}
  \lesssim_{\zeta,\sigma,j}S_{\sigma,j}
\end{equation}
holds uniformly for $0\le t\le t_*$;
\item $\ubar$ solves \eqref{eq:approx-equation}, and its residual satisfies
\begin{equation}\label{eq:full-residual-package}
  \norm{L^j}{\nabla^kE(t)}
  \lesssim_{\zeta,k,j}S_{k,j}A\nu;
\end{equation}
\item as $\mu\to\infty$ with $\zeta$ fixed,
\begin{equation}\label{eq:scale-separation-package}
  t_*\mu^{2\alpha}\to0,
  \qquad
  \frac{\mu^{2\alpha}}{A\nu}\to0,
  \qquad
  t_*\to0.
\end{equation}
\end{enumerate}
In the Besov case, $C_{\mathrm{in}}(X)$ may depend on the fixed index $q$.
\end{proposition}

\subsection{Proof of Proposition \ref{prop:approximation-package}}

\begin{lemma}[Divergence-free correction]\label{lem:div-free}
The field $u_0$ in \eqref{eq:u0} is divergence free.
\end{lemma}

\begin{proof}
Set $\psi=A\mu^{-1}f(\mu\rho)$.  Then
\[
  (u_{0,r},u_{0,z})=(-\partial_z\psi,\partial_r\psi),
\]
so
\[
  \partial_ru_{0,r}+\partial_zu_{0,z}=0.
\]
Moreover,
\[
  \frac{u_{0,r}}r
  =-A\frac{f'(\mu\rho)\partial_z\rho}{r},
  \qquad
  \partial_zu_c
  =A\frac{f'(\mu\rho)\partial_z\rho}{r}.
\]
Since the field is axisymmetric,
\[
  \divg u_0
  =\partial_ru_{0,r}+\frac{u_{0,r}}r+\partial_z(u_{0,z}+u_c)=0.
\]
\end{proof}

\begin{lemma}[Initial profile bounds and smallness]\label{lem:initial-smallness}
For every $\sigma\ge0$ and $1<j<\infty$,
\begin{equation}\label{eq:initial-W-profile-bound}
  \norm{W^{\sigma,j}}{u_0}
  \lesssim_{\sigma,j}S_{\sigma,j}.
\end{equation}
For every $\sigma>0$, $1<j<\infty$, and $1\le q\le\infty$,
\begin{equation}\label{eq:initial-B-profile-bound}
  \norm{B^\sigma_{j,q}}{u_0}
  \lesssim_{\sigma,j,q}S_{\sigma,j}.
\end{equation}
In particular, for either fixed target space
\[
  X=W^{s,p}
  \qquad\text{or}\qquad
  X=B^s_{p,q}\ \text{with $q$ fixed},
\]
there is a constant $C_{\mathrm{in}}(X)>0$ such that
\begin{equation}
  \norm{X}{u_0}\le C_{\mathrm{in}}(X)\zeta^2.
\end{equation}
The constant $C_{\mathrm{in}}(X)$ depends only on the fixed exponents and profiles; in the Besov case it is allowed to depend on the already fixed index $q$.
\end{lemma}

\begin{proof}
The three principal components in \eqref{eq:initial-components} are toroidal profiles of amplitude $A$.  The correction $u_c$ has amplitude $A\nu/\mu=A\vartheta$ and is therefore smaller.  Lemma \ref{lem:toroidal-calculus} gives, at every integer order $M\ge0$,
\[
  \norm{W^{M,j}}{u_0}\lesssim_{M,j}A\mu^MV_j.
\]
Interpolation between two neighboring integer orders proves \eqref{eq:initial-W-profile-bound}, including $\sigma=0$, where $W^{0,j}=L^j$.  If $\sigma>0$, choose an integer $M>\sigma$.  Real interpolation between $L^j$ and $W^{M,j}$, equivalently the standard finite-difference profile estimate, gives
\[
  \norm{B^\sigma_{j,q}}{u_0}
  \lesssim_{\sigma,j,q}
  \norm{L^j}{u_0}^{1-\sigma/M}
  \norm{W^{M,j}}{u_0}^{\sigma/M}
  \lesssim_{\sigma,j,q}A\mu^\sigma V_j,
\]
which proves \eqref{eq:initial-B-profile-bound}.  Taking $\sigma=s$ and $j=p$, and then using \eqref{eq:target-scale}, proves \eqref{eq:initial-target-smallness}.  No $B^0_{j,q}$ estimate for arbitrary $q$ is asserted.
\end{proof}

\begin{lemma}[Scale separation]\label{lem:scale-separation}
For fixed $\zeta$ and $\mu\to\infty$,
\begin{equation}\label{eq:time-separation}
  t_*\mu^{2\alpha}
  =\zeta^{-N-2}\mu^{-\delta+b/p}
  \longrightarrow0,
\end{equation}
and
\begin{equation}\label{eq:residual-separation}
  \frac{\mu^{2\alpha}}{A\nu}
  =\zeta^{-2}\mu^{-\delta+b(1+1/p)}
  \longrightarrow0.
\end{equation}
Moreover, $t_*\to0$ as $\mu\to\infty$.
\end{lemma}

\begin{proof}
The first two identities follow directly from \eqref{eq:delta-b}, \eqref{eq:amplitude-scales}, and \eqref{eq:tstar-Kstar}.  Since $b=\delta/100$, both exponents of $\mu$ are strictly negative.  Finally,
\[
  L=\zeta^2\mu^{2\alpha+\delta-b/p},
\]
whose exponent is positive, and hence $t_*=\zeta^{-N}L^{-1}\to0$.
\end{proof}

\begin{lemma}[Uniform approximate-profile bounds]\label{lem:approx-profile-bounds}
For every $\sigma\ge0$, $1<j<\infty$, and $0\le t\le t_*$,
\begin{equation}
  \norm{W^{\sigma,j}}{\ubar(t)}
  \lesssim_{\zeta,\sigma,j}S_{\sigma,j}.
\end{equation}
In particular,
\begin{equation}\label{eq:approx-lipschitz}
  \norm{L^\infty}{\nabla\ubar(t)}
  \le C_{\mathrm{app}}(\zeta)L,
  \qquad 0\le t\le t_*.
\end{equation}
\end{lemma}

\begin{proof}
By Lemma \ref{lem:phase-upper-prelim}, the derivatives of the phase on $[0,t_*]$ cost at most $C_\zeta\mu$ per derivative.  Lemma \ref{lem:toroidal-calculus} then yields \eqref{eq:approx-profile-scale}.  The Lipschitz estimate follows either directly or by taking an integer order above $1+3/j$ and applying Sobolev embedding.
\end{proof}

\begin{lemma}[Geometric residual]\label{lem:geometric-residual}
For every integer $k\ge0$, every $1<j<\infty$, and $0\le t\le t_*$,
\begin{equation}\label{eq:geometric-residual-bound}
  \norm{L^j}{\nabla^kE_{\mathrm{geom}}(t)}
  \lesssim_{\zeta,k,j}S_{k,j}A\nu.
\end{equation}
\end{lemma}

\begin{proof}
Each term in \eqref{eq:geometric-residual-components} contains either a factor $r^{-1}$ or the correction $u_c$.  The former contributes $\nu$, while the latter has size $A\nu/\mu=A\vartheta$.  A derivative of a principal profile costs at most $C_\zeta\mu$ on $[0,t_*]$.  Thus every Leibniz term has size at most $C_\zeta A^2\nu\mu^k$ on a torus of volume $\simeq\mu^{-2}\nu^{-1}$.  This is exactly \eqref{eq:geometric-residual-bound}.
\end{proof}

\begin{lemma}[Full-space fractional dissipative residual]\label{lem:dissipative-residual}
For every integer $k\ge0$, every $1<j<\infty$, and $0\le t\le t_*$,
\begin{equation}\label{eq:dissipative-residual-bound}
  \norm{L^j}{\nabla^kE_{\mathrm{diss}}(t)}
  \lesssim_{\zeta,k,j}S_{k,j}\mu^{2\alpha}.
\end{equation}
Consequently, for fixed $\zeta$,
\begin{equation}\label{eq:dissipative-relative-bound}
  \norm{L^j}{\nabla^kE_{\mathrm{diss}}(t)}
  \le o_{\mu\to\infty}(1)\,S_{k,j}A\nu.
\end{equation}
The estimates are over all of $\R^3$; no compact-support assertion is made for $E_{\mathrm{diss}}$.
\end{lemma}

\begin{proof}
Since $E_{\mathrm{diss}}=\Lambdaop^{2\alpha}\ubar$ and $\nabla^k\Lambdaop^{-k}$ is a finite combination of zero-order Riesz multipliers,
\begin{equation}\label{eq:dissipative-riesz-step}
  \norm{L^j}{\nabla^kE_{\mathrm{diss}}}
  \lesssim_{k,j}
  \norm{L^j}{\Lambdaop^{k+2\alpha}\ubar}.
\end{equation}
Choose any integer $m>k+2\alpha$.  Interpolation between order $0$ and order $m$, followed by Lemma \ref{lem:approx-profile-bounds}, gives
\begin{equation}\label{eq:dissipative-interpolation-step}
\begin{aligned}
  \norm{L^j}{\Lambdaop^{k+2\alpha}\ubar}
  &\lesssim
  \norm{L^j}{\ubar}^{1-(k+2\alpha)/m}
  \norm{W^{m,j}}{\ubar}^{(k+2\alpha)/m}\\
  &\lesssim_{\zeta,k,j}
  (AV_j)^{1-(k+2\alpha)/m}
  (A\mu^mV_j)^{(k+2\alpha)/m}\\
  &=A\mu^{k+2\alpha}V_j
  =S_{k,j}\mu^{2\alpha}.
\end{aligned}
\end{equation}
This proves \eqref{eq:dissipative-residual-bound}.  Dividing by $S_{k,j}A\nu$ and using \eqref{eq:residual-separation} proves \eqref{eq:dissipative-relative-bound}.  The use of the global multiplier $\Lambdaop^{2\alpha}$ in \eqref{eq:dissipative-riesz-step} explicitly controls the nonlocal tail throughout $\R^3$.
\end{proof}

\begin{proof}[Proof of Proposition \ref{prop:approximation-package}]
The divergence-free property follows from Lemma \ref{lem:div-free}, and the initial estimate follows from Lemma \ref{lem:initial-smallness}.  Lemma \ref{lem:approx-profile-bounds} gives \eqref{eq:approx-profile-scale}.  Combining Lemmas \ref{lem:geometric-residual} and \ref{lem:dissipative-residual}, and then taking $\mu$ large enough that the relative factor in \eqref{eq:dissipative-relative-bound} is at most one, gives \eqref{eq:full-residual-package}.  The remaining assertions follow from Lemma \ref{lem:scale-separation}.
\end{proof}

\section{Norm growth of the approximate solution}\label{sec:growth}

The main result of this section is Proposition \ref{prop:approx-growth}.  We first identify a fixed region on which the transported phase is nondegenerate.

The phase at the inflation time is
\begin{equation}\label{eq:phase-star}
  \Phi_*:=\Phi(t_*,\rho,\varphi)
  =\varphi-\frac{t_*A}{\rho}.
\end{equation}
A direct calculation gives
\begin{equation}\label{eq:phase-z-derivative}
  \partial_z\Phi_*
  =\frac{\cos\varphi}{\rho}
  +\frac{t_*A}{\rho^2}\sin\varphi.
\end{equation}
Choose a closed interval $J_\varphi\Subset(0,\pi)$ on which $\sin\varphi\ge c_\varphi>0$.  On the set
\begin{equation}\label{eq:nondegenerate-region}
  \mathcal R_*:=\{\mu\rho\in I_g,\ \varphi\in J_\varphi\},
\end{equation}
if $K_*/\mu=\zeta^{-N}$ exceeds a fixed threshold, then
\begin{equation}\label{eq:phase-nondegenerate}
  cK_*\le\abs{\partial_z\Phi_*}\le CK_*.
\end{equation}
Moreover, by Lemma \ref{lem:phase-upper-prelim},
\begin{equation}\label{eq:phase-higher-star}
  \abs{\nabla^j\Phi_*}\lesssim_jK_*\mu^{j-1},
  \qquad j\ge1.
\end{equation}

\begin{proposition}[Growth of the approximate solution]\label{prop:approx-growth}
Let $s>0$ and $1<p<\infty$.
\begin{enumerate}[label=\textup{(\roman*)},leftmargin=2.2em,itemsep=3pt]
\item For every $1\le q\le\infty$, there exists $C_{\mathrm{osc}}^B(s,p,q)\ge1$ such that, if $\zeta^{-N}\ge C_{\mathrm{osc}}^B(s,p,q)$, then
\begin{align}
  \norm{\dot B^s_{p,q}}{\ubar(t_*)}
  &\ge c_{\dot B}(s,p,q)AK_*^sV_p,
  \label{eq:approx-hom-besov-lower}\\
  \norm{B^s_{p,q}}{\ubar(t_*)}
  &\ge c_B(s,p,q)AK_*^sV_p.
  \label{eq:approx-inhom-besov-lower}
\end{align}
\item There exists $C_{\mathrm{osc}}^W(s,p)\ge1$ such that, if $\zeta^{-N}\ge C_{\mathrm{osc}}^W(s,p)$, then
\begin{align}
  \norm{\dot W^{s,p}}{\ubar(t_*)}
  &\ge c_{\dot W}(s,p)AK_*^sV_p,
  \label{eq:approx-hom-sobolev-lower}\\
  \norm{W^{s,p}}{\ubar(t_*)}
  &\ge c_W(s,p)AK_*^sV_p.
  \label{eq:approx-inhom-sobolev-lower}
\end{align}
\end{enumerate}
In every case,
\begin{equation}\label{eq:approx-growth-scale}
  AK_*^sV_p=\zeta^{2-Ns}.
\end{equation}
\end{proposition}

\subsection{A weighted non-cancellation lemma}

\begin{lemma}[Oscillatory non-cancellation]\label{lem:noncancellation}
Let $1\le p<\infty$ and let $\beta\in\R$ be fixed.  Set
\[
  \mathfrak m_p:=\frac1{2\pi}\int_0^{2\pi}\abs{\sin\tau}^p\dd\tau>0.
\]
There is a constant $C_{p,C_0}>0$ with the following property.  Let $I$ be a bounded interval, let $a\in W^{1,1}(I)$ be nonnegative, and let $\phi\in C^2(I)$ be strictly monotone with
\[
  \inf_I\abs{\phi'}\ge\lambda,
  \qquad
  \sup_I\abs{\phi''}\le C_0\lambda\abs I^{-1}.
\]
Then
\begin{equation}\label{eq:noncancellation-general}
\begin{aligned}
  \int_Ia(y)\abs{\sin(\phi(y)+\beta)}^p\dd y
  &\ge \mathfrak m_p\int_Ia(y)\dd y\\
  &\quad-\frac{C_{p,C_0}}{\lambda}
  \left(\norm{L^1(I)}{a'}+\abs I^{-1}\norm{L^1(I)}{a}\right).
\end{aligned}
\end{equation}
The same conclusion holds with sine replaced by cosine.
\end{lemma}

\begin{proof}
Let $J:=\phi(I)$, viewed as an unoriented interval, let $y=y(\tau)$ be the inverse map, and put
\[
  b(\tau):=\frac{a(y(\tau))}{\abs{\phi'(y(\tau))}}.
\]
Then $b\ge0$, $b\in W^{1,1}(J)$, and the change of variables gives
\begin{equation}\label{eq:noncancellation-change-variable}
  \int_Ia(y)\abs{\sin(\phi(y)+\beta)}^p\dd y
  =\int_Jb(\tau)\abs{\sin(\tau+\beta)}^p\dd\tau,
  \qquad
  \int_Jb=\int_Ia.
\end{equation}
Let $G_{p,\beta}$ be the bounded $2\pi$-periodic primitive of
\[
  \abs{\sin(\tau+\beta)}^p-\mathfrak m_p
\]
with zero mean.  Integration by parts on $J$ and the one-dimensional estimate
\[
  \norm{L^\infty(J)}{b}
  \le \abs J^{-1}\norm{L^1(J)}{b}+\norm{L^1(J)}{b'}
\]
yield
\begin{equation}\label{eq:periodic-primitive-bound}
  \left|\int_Jb(\tau)
  \bigl(\abs{\sin(\tau+\beta)}^p-\mathfrak m_p\bigr)\dd\tau\right|
  \le C_p\left(\norm{L^1(J)}{b'}+\abs J^{-1}\norm{L^1(J)}{b}\right).
\end{equation}
Writing $g_\phi(y):=\abs{\phi'(y)}$, differentiating $b=a/g_\phi$, and using $\dd\tau=g_\phi(y)\dd y$, we obtain
\begin{align}
  \norm{L^1(J)}{b'}
  &=\int_I\left|\frac{\dd}{\dd y}\left(\frac{a}{g_\phi}\right)\right|\dd y\notag\\
  &\le \frac1\lambda\norm{L^1(I)}{a'}
  +\frac{\sup_I\abs{\phi''}}{\lambda^2}\norm{L^1(I)}{a}\notag\\
  &\le \frac1\lambda\norm{L^1(I)}{a'}
  +\frac{C_0}{\lambda\abs I}\norm{L^1(I)}{a}.
  \label{eq:transformed-weight-variation}
\end{align}
Moreover, $\abs J\ge\lambda\abs I$ and $\norm{L^1(J)}{b}=\norm{L^1(I)}{a}$.  Substituting these bounds into \eqref{eq:periodic-primitive-bound} and then using \eqref{eq:noncancellation-change-variable} proves \eqref{eq:noncancellation-general}.  The cosine case is a fixed phase shift.
\end{proof}

\subsection{Oscillatory derivative growth}

\begin{lemma}[Integer derivative bounds]\label{lem:integer-derivative-bounds}
Let $m\ge1$ be a fixed integer and $1<p<\infty$.  There exists $C_{\mathrm{osc}}(m,p)\ge1$, depending only on $m,p$ and the fixed profiles, such that if
\begin{equation}\label{eq:osc-threshold-integer}
  \zeta^{-N}\ge C_{\mathrm{osc}}(m,p),
\end{equation}
then
\begin{equation}\label{eq:integer-lower}
  \norm{L^p}{\partial_z^m\ubar_\theta(t_*)}
  \ge c_{m,p}AK_*^mV_p.
\end{equation}
For every $1<j<\infty$,
\begin{equation}\label{eq:integer-upper}
  \norm{L^j}{\nabla^m\ubar(t_*)}
  \le C_{m,j}A(\mu+K_*)^mV_j.
\end{equation}
In particular, when $K_*/\mu\ge1$, the right-hand side is bounded by $C_{m,j}AK_*^mV_j$.
\end{lemma}

\begin{proof}
\medskip\noindent\textit{Step 1: derivative expansion.}\quad Repeated differentiation of the swirl gives
\begin{equation}\label{eq:main-derivative-expansion}
  \partial_z^m\ubar_\theta(t_*)
  =Ag(\mu\rho)(\partial_z\Phi_*)^m
  \sin\left(\Phi_*+\frac{m\pi}{2}\right)
  +\mathcal E_m.
\end{equation}
Every term in $\mathcal E_m$ either contains a derivative of $g(\mu\rho)$ or contains at least one phase derivative of order at least two.  Using \eqref{eq:phase-higher-star} and the product rule, one obtains
\begin{equation}\label{eq:derivative-remainder}
  \norm{L^p}{\mathcal E_m}
  \le C_{m,p}AK_*^{m-1}\mu V_p.
\end{equation}

\medskip\noindent\textit{Step 2: a fixed toroidal box.}\quad We now prove the lower bound for the principal term without suppressing the fibre geometry.  Choose
\[
  \varrho_0\in\operatorname{int}I_g,
  \qquad
  \varphi_0\in\operatorname{int}J_\varphi,
\]
and set $\xi_0=\varrho_0\cos\varphi_0$, $y_0=\varrho_0\sin\varphi_0$.  Since $y_0>0$ and the set
\[
  \left\{(\xi,y):
  \sqrt{\xi^2+y^2}\in\operatorname{int}I_g,
  \ \arg(\xi+iy)\in\operatorname{int}J_\varphi\right\}
\]
is open, there are fixed compact intervals $I_\xi$ and $I_y$, of positive length and containing $\xi_0$ and $y_0$, such that
\begin{equation}\label{eq:scaled-fibre-rectangle}
  \mathcal Q:=I_\xi\times I_y
  \Subset
  \left\{(\xi,y):
  \sqrt{\xi^2+y^2}\in I_g,
  \ \arg(\xi+iy)\in J_\varphi\right\}.
\end{equation}
In particular, $y\ge c_{\mathcal Q}>0$ and
$\varrho(\xi,y):=\sqrt{\xi^2+y^2}$ is bounded above and below on $\mathcal Q$.

Put
\begin{equation}\label{eq:scaled-toroidal-variables-growth}
  R:=\nu^{-1},
  \qquad
  \xi:=\mu(r-R),
  \qquad
  y:=\mu z,
  \qquad
  \lambda_*:=\zeta^{-N}=\frac{K_*}{\mu}.
\end{equation}
Then $\mu\rho=\varrho(\xi,y)$ and, on $\mathcal Q$,
\begin{equation}\label{eq:scaled-phase-growth}
  \Phi_*(r,z)
  =\widetilde\Phi_{\lambda_*}(\xi,y)
  :=\arg(\xi+iy)-\frac{\lambda_*}{\varrho(\xi,y)}.
\end{equation}
A direct differentiation gives
\begin{align}
  \partial_y\widetilde\Phi_{\lambda_*}
  &=\frac{\xi}{\varrho^2}+\lambda_*\frac{y}{\varrho^3},
  \label{eq:scaled-phase-first}\\
  \partial_y^2\widetilde\Phi_{\lambda_*}
  &=-\frac{2\xi y}{\varrho^4}
  +\lambda_*\frac{\xi^2-2y^2}{\varrho^5}.
  \label{eq:scaled-phase-second}
\end{align}
Consequently, there are constants $c_{\mathcal Q},C_{\mathcal Q}>0$ and a threshold $\lambda_0\ge1$, depending only on the fixed rectangle, such that whenever $\lambda_*\ge\lambda_0$,
\begin{equation}\label{eq:scaled-phase-uniform-bounds}
  c_{\mathcal Q}\lambda_*
  \le \partial_y\widetilde\Phi_{\lambda_*}(\xi,y)
  \le C_{\mathcal Q}\lambda_*,
  \qquad
  \abs{\partial_y^2\widetilde\Phi_{\lambda_*}(\xi,y)}
  \le C_{\mathcal Q}\lambda_*
\end{equation}
uniformly on $\mathcal Q$.  Notice also that
\begin{equation}\label{eq:z-y-phase-derivative}
  \partial_z\Phi_*=\mu\partial_y\widetilde\Phi_{\lambda_*}.
\end{equation}

\medskip\noindent\textit{Step 3: fibrewise non-cancellation.}\quad For each fixed $\xi\in I_\xi$, define the nonnegative weight
\begin{equation}\label{eq:fibre-weight}
  a_{\xi,\lambda_*}(y)
  :=g(\varrho(\xi,y))^p
  \abs{\partial_y\widetilde\Phi_{\lambda_*}(\xi,y)}^{mp},
  \qquad y\in I_y.
\end{equation}
Because $g\ge c_g>0$ on $I_g$, and because of \eqref{eq:scaled-phase-uniform-bounds},
\begin{equation}\label{eq:fibre-weight-size-variation}
  c\lambda_*^{mp}\le a_{\xi,\lambda_*}(y)
  \le C\lambda_*^{mp},
  \qquad
  \abs{\partial_ya_{\xi,\lambda_*}(y)}
  \le C a_{\xi,\lambda_*}(y),
\end{equation}
with constants uniform in $\xi$, $y$, and $\lambda_*\ge\lambda_0$.  Indeed, after taking a logarithmic derivative, the only factors are $g'(\varrho)/g(\varrho)$, $\partial_y\varrho$, and
$\partial_y^2\widetilde\Phi_{\lambda_*}/
\partial_y\widetilde\Phi_{\lambda_*}$, all uniformly bounded on $\mathcal Q$.

Apply Lemma \ref{lem:noncancellation} on the fixed interval $I_y$ with phase
$\widetilde\Phi_{\lambda_*}(\xi,\cdot)$ and shift $m\pi/2$.  By
\eqref{eq:scaled-phase-uniform-bounds}--\eqref{eq:fibre-weight-size-variation}, its error is at most $C\lambda_*^{-1}$ times the weight mass.  Enlarging $\lambda_0$ if necessary therefore gives, uniformly for $\xi\in I_\xi$,
\begin{equation}\label{eq:fibre-noncancellation-lower}
\begin{aligned}
  &\int_{I_y}a_{\xi,\lambda_*}(y)
  \abs{\sin\left(\widetilde\Phi_{\lambda_*}(\xi,y)
  +\frac{m\pi}{2}\right)}^p\dd y\\
  &\hspace{4cm}\ge c_{m,p,\mathcal Q}\lambda_*^{mp}.
\end{aligned}
\end{equation}

\medskip\noindent\textit{Step 4: the three-dimensional measure.}\quad It remains to restore the three-dimensional Jacobian.  On the physical sub-torus corresponding to
$\theta\in[0,2\pi)$ and $(\xi,y)\in\mathcal Q$,
\[
  r=R+\frac\xi\mu,
  \qquad
  \dd x=r\,\dd r\,\dd\theta\,\dd z
  =\left(R+\frac\xi\mu\right)\mu^{-2}\dd\xi\,\dd y\,\dd\theta.
\]
Since $\mu/\nu\to\infty$, for all sufficiently large $\mu$ one has
$R+\xi/\mu\ge\frac12\nu^{-1}$ on $I_\xi$.  Combining this fact with
\eqref{eq:z-y-phase-derivative} and \eqref{eq:fibre-noncancellation-lower}, we obtain
\begin{align}
  &\norm{L^p}{Ag(\mu\rho)(\partial_z\Phi_*)^m
  \sin\left(\Phi_*+\frac{m\pi}{2}\right)}^p\notag\\
  &\quad\ge cA^p\mu^{mp-2}\nu^{-1}
  \int_{I_\xi}\int_{I_y}a_{\xi,\lambda_*}(y)
  \abs{\sin\left(\widetilde\Phi_{\lambda_*}(\xi,y)
  +\frac{m\pi}{2}\right)}^p\dd y\dd\xi\notag\\
  &\quad\ge c_{m,p}A^p(\lambda_*\mu)^{mp}\mu^{-2}\nu^{-1}
  =c_{m,p}A^pK_*^{mp}V_p^p.
  \label{eq:main-term-lower-pth-power}
\end{align}
Taking the $p$th root proves
\begin{equation}\label{eq:main-term-lower}
  \norm{L^p}{Ag(\mu\rho)(\partial_z\Phi_*)^m
  \sin\left(\Phi_*+\frac{m\pi}{2}\right)}
  \ge c_{m,p}AK_*^mV_p.
\end{equation}
The ratio of \eqref{eq:derivative-remainder} to the right-hand side of
\eqref{eq:main-term-lower} is $O(\mu/K_*)=O(\lambda_*^{-1})$.  Increasing the fixed threshold $C_{\mathrm{osc}}(m,p)$ proves \eqref{eq:integer-lower}.  Finally, the chain rule, \eqref{eq:phase-higher-star}, and Lemma \ref{lem:toroidal-calculus} give directly
\[
  \norm{L^j}{\nabla^m\ubar(t_*)}
  \le C_{m,j}A(\mu+K_*)^mV_j,
\]
which proves \eqref{eq:integer-upper}.
\end{proof}

\begin{proof}[Proof of Proposition \ref{prop:approx-growth}]
\textit{Step 1: the Besov lower bound.}
Let $m=\lfloor s\rfloor+1$, so $m>s$.  Iterating the fundamental theorem of calculus gives
\begin{equation}\label{eq:fd-integral-formula}
  \Delta_{he_z}^mf(x)
  =\int_{[0,h]^m}\partial_z^mf\bigl(x+(t_1+\cdots+t_m)e_z\bigr)
  \dd t_1\cdots\dd t_m.
\end{equation}
Subtracting $h^m\partial_z^mf(x)$ and applying the fundamental theorem once more yields
\begin{equation}\label{eq:fd-expansion}
  \Delta_{he_z}^m\ubar_\theta(t_*)
  =h^m\partial_z^m\ubar_\theta(t_*)+R_h,
  \qquad
  \norm{L^p}{R_h}
  \le C_mh^{m+1}\norm{L^p}{\partial_z^{m+1}\ubar_\theta(t_*)}.
\end{equation}
Take
\[
  h\in[c_0K_*^{-1},2c_0K_*^{-1}],
\]
where $c_0>0$ is fixed and small.  Lemma \ref{lem:integer-derivative-bounds} at orders $m$ and $m+1$ gives
\[
  h^m\norm{L^p}{\partial_z^m\ubar_\theta(t_*)}
  \ge c_mc_0^mAV_p,
  \qquad
  \norm{L^p}{R_h}\le C_m'c_0^{m+1}AV_p.
\]
Choose $c_0$ so that the second term is at most one half of the first.  Then
\begin{equation}\label{eq:fd-lower}
  \norm{L^p}{\Delta_{he_z}^m\ubar_\theta(t_*)}
  \gtrsim AV_p
\end{equation}
throughout this interval of $h$.

Translations in the $z$ direction leave $e_r,e_\theta,e_z$ unchanged.  The cylindrical components of $\Delta_{he_z}^m\ubar$ are therefore pointwise orthogonal, and
\[
  \abs{\Delta_{he_z}^m\ubar}
  \ge\abs{\Delta_{he_z}^m\ubar_\theta}.
\]
Restricting \eqref{eq:directional-besov} to $h\simeq K_*^{-1}$ gives
\[
  \norm{\dot B^s_{p,q}}{\ubar(t_*)}
  \ge c_{\dot B}(s,p,q)K_*^sAV_p.
\]
For $q=\infty$, use any single $h$ in the chosen interval.  For $s>0$, the standard comparison $\norm{\dot B^s_{p,q}}{f}\lesssim_{s,p,q}\norm{B^s_{p,q}}{f}$ follows by controlling the negative dyadic blocks with $\norm{L^p}{f}$.  Consequently,
\[
  \norm{B^s_{p,q}}{\ubar(t_*)}
  \ge c_B(s,p,q)K_*^sAV_p.
\]
This proves the two Besov estimates.  Finally,
\[
  AK_*^sV_p
  =A(\zeta^{-N}\mu)^s\mu^{-2/p}\nu^{-1/p}
  =\zeta^{2-Ns}.
\]

\medskip
\noindent\textit{Step 2: the Sobolev lower bound.}
Let $m=\lfloor s\rfloor+1$ and $n=m+1$.  Since
\[
  \partial_z^m=R_3^m\Lambdaop^m
\]
and Riesz transforms are bounded on $L^p$, Lemma \ref{lem:integer-derivative-bounds} gives
\begin{equation}\label{eq:Wm-lower}
  \norm{\dot W^{m,p}}{\ubar(t_*)}
  \gtrsim\norm{L^p}{\partial_z^m\ubar_\theta(t_*)}
  \gtrsim AK_*^mV_p.
\end{equation}
The upper estimate at order $n$ is
\begin{equation}\label{eq:Wn-upper}
  \norm{\dot W^{n,p}}{\ubar(t_*)}
  \lesssim AK_*^nV_p,
\end{equation}
where the constant is uniform once the prescribed separation threshold is fixed.  Choose $\theta\in(0,1)$ so that
\[
  m=(1-\theta)s+\theta n,
  \qquad
  \theta=\frac{m-s}{n-s}.
\]
By \eqref{eq:sobolev-interpolation-prelim},
\[
  \norm{\dot W^{m,p}}{\ubar}
  \lesssim
  \norm{\dot W^{s,p}}{\ubar}^{1-\theta}
  \norm{\dot W^{n,p}}{\ubar}^{\theta}.
\]
Combining this with \eqref{eq:Wm-lower}--\eqref{eq:Wn-upper} and solving for the order-$s$ norm gives
\[
  \norm{\dot W^{s,p}}{\ubar(t_*)}
  \ge c_{\dot W}(s,p)AK_*^sV_p.
\]
The zero-order multiplier $\Lambdaop^sJ^{-s}$ is bounded on $L^p$, so $\norm{\dot W^{s,p}}{f}\lesssim_{s,p}\norm{W^{s,p}}{f}$.  Hence
\[
  \norm{W^{s,p}}{\ubar(t_*)}
  \ge c_W(s,p)AK_*^sV_p.
\]
This proves the two Sobolev estimates.
\end{proof}

\section{Stability of the approximation}\label{sec:stability}

Fix the target exponents from Theorem \ref{thm:main}.  Choose an integer
\begin{equation}\label{eq:Kb-choice}
  K_{\mathrm b}\ge \max\left\{6,\left\lceil s+d_p\right\rceil+6\right\},
\end{equation}
and set
\begin{equation}\label{eq:gamma-kappa-choice}
  \gamma_k:=1-\frac{k}{4(K_{\mathrm b}+1)},
  \qquad 0\le k\le K_{\mathrm b},
  \qquad
  \gamma_*:=\gamma_{K_{\mathrm b}},
  \qquad
  \kappa:=\frac1{8(K_{\mathrm b}+1)}.
\end{equation}
Then $\gamma_*>3/4$.  Fix $\ell>\max\{p,2\}$ and $\eta>0$ so that
\begin{equation}\label{eq:ell-eta-choice}
  \frac1\ell<\frac\kappa4,
  \qquad
  0<\eta<\frac{b\kappa}{4}.
\end{equation}
Finally, define
\begin{equation}\label{eq:beta-p}
  \beta_p:=\gamma_*-\left(\frac12-\frac1p\right)>0.
\end{equation}

Let $u$ be the smooth solution issued from $u_0$ on its maximal interval and set
\begin{equation}\label{eq:w-def}
  w:=u-\ubar.
\end{equation}
Subtracting \eqref{eq:approx-equation} from \eqref{eq:fns} gives
\begin{equation}\label{eq:error-equation}
\begin{cases}
  \partial_tw+(-\Delta)^\alpha w
  +u\cdot\nabla w+w\cdot\nabla\ubar+\nabla\pi_w=-E,\\
  \divg w=0,\\
  w|_{t=0}=0.
\end{cases}
\end{equation}

\begin{proposition}[Stability of the approximation]\label{prop:stability}
For each fixed sufficiently small $\zeta$, there exists $\mu_{\mathrm{stab}}(\zeta)$ such that, for $\mu\ge\mu_{\mathrm{stab}}(\zeta)$, the exact solution exists smoothly on $[0,t_*]$ and
\begin{align}
  \norm{W^{s,p}}{u(t_*)-\ubar(t_*)}
  &\lesssim_\zeta S_{s,p}\vartheta^{\beta_p},
  \label{eq:W-target-error}\\
  \norm{B^s_{p,q}}{u(t_*)-\ubar(t_*)}
  &\lesssim_{\zeta,s,p,q,\eta}
  S_{s,p}\mu^\eta\vartheta^{\beta_p}
  \label{eq:B-target-error}
\end{align}
for every fixed $1\le q\le\infty$; the implicit constant in the Besov estimate may depend on that fixed $q$.  In particular, each right-hand side is $o_{\mu\to\infty}(S_{s,p})$ with $\zeta$ fixed.
\end{proposition}

We prove the proposition through a Lipschitz bootstrap.  Let $C_{\mathrm{app}}(\zeta)$ be the constant in \eqref{eq:approx-lipschitz}.  Define $T_{\mathrm{boot}}$ to be the supremum of all $T\le\min\{t_*,T_{\max}\}$ such that
\begin{equation}\label{eq:lipschitz-bootstrap}
  \norm{L^\infty}{\nabla u(t)}
  \le2C_{\mathrm{app}}(\zeta)L,
  \qquad 0\le t\le T.
\end{equation}
By continuity and the initial profile bounds, $T_{\mathrm{boot}}>0$ after increasing $C_{\mathrm{app}}(\zeta)$ if necessary.

\subsection{Bootstrap and finite-exponent estimates}

For a divergence-free vector field $v$, define
\begin{equation}\label{eq:projected-commutator}
  \mathcal C_\sigma(v)
  :=\PP J^\sigma\nabla\cdot(v\otimes v)-v\cdot\nabla J^\sigma v.
\end{equation}

\begin{lemma}[Projected commutator]\label{lem:projected-commutator}
Let $1<j<\infty$ and $\sigma\ge1$.  For every smooth divergence-free $v$,
\begin{equation}\label{eq:projected-commutator-bound}
  \norm{L^j}{\mathcal C_\sigma(v)}
  \le C_{\sigma,j}\norm{L^\infty}{\nabla v}
  \norm{W^{\sigma,j}}{v}.
\end{equation}
\end{lemma}

\begin{proof}
Since $v$ is divergence free and $J^\sigma$ commutes with the Leray projector,
\begin{equation}\label{eq:commutator-splitting}
  \mathcal C_\sigma(v)
  =\PP[J^\sigma,v\cdot\nabla]v
  +[\PP,v\cdot\nabla]J^\sigma v.
\end{equation}
The Kato--Ponce commutator estimate \cite{KatoPonce1988} gives
\begin{equation}\label{eq:KP-part}
  \norm{L^j}{[J^\sigma,v\cdot\nabla]v}
  \lesssim_{\sigma,j}
  \norm{L^\infty}{\nabla v}\norm{W^{\sigma,j}}{v}.
\end{equation}

For the second term, write $\PP_{ab}=\delta_{ab}+R_aR_b$ and set $F=J^\sigma v$.  Componentwise, $[\PP,v\cdot\nabla]F$ is a finite sum of terms of the form
\[
  [R_aR_b,v_c]\partial_cF_b.
\]
The first Calder\'on commutator estimate \cite{Calderon1965,CoifmanMeyer1978} therefore yields
\begin{equation}\label{eq:Calderon-part}
  \norm{L^j}{[R_aR_b,v_c]\partial_cF_b}
  \lesssim_j
  \norm{L^\infty}{\nabla v}\norm{L^j}{F}.
\end{equation}
Combining \eqref{eq:commutator-splitting}--\eqref{eq:Calderon-part} with the $L^j$ boundedness of $\PP$ proves \eqref{eq:projected-commutator-bound}.
\end{proof}

\begin{proposition}[Finite-$L^j$ propagation]\label{prop:finite-ell}
Assume that a smooth solution $u$ of \eqref{eq:fns} satisfies
\begin{equation}\label{eq:Lipschitz-assumption-general}
  \norm{L^\infty}{\nabla u(t)}\le M,
  \qquad 0\le t\le T.
\end{equation}
Then, for every $1<j<\infty$ and every real $\sigma\ge1$,
\begin{equation}\label{eq:finite-ell-propagation}
  \norm{W^{\sigma,j}}{u(t)}
  \le C_{\sigma,j}\norm{W^{\sigma,j}}{u_0}
  \exp(C_{\sigma,j}Mt),
  \qquad 0\le t\le T.
\end{equation}
Consequently, under the bootstrap \eqref{eq:lipschitz-bootstrap}, for every fixed real $\sigma\ge1$,
\begin{equation}\label{eq:finite-ell-scale}
  \norm{W^{\sigma,\ell}}{u(t)}
  +\norm{W^{\sigma,\ell}}{\ubar(t)}
  \lesssim_{\zeta,\sigma,\ell}S_{\sigma,\ell},
  \qquad 0\le t\le T_{\mathrm{boot}}.
\end{equation}
\end{proposition}

\begin{proof}
Write the exact equation in projected form and set $F:=J^\sigma u$:
\begin{equation}\label{eq:F-equation}
  \partial_tF+u\cdot\nabla F+\Lambdaop^{2\alpha}F
  =-\mathcal C_\sigma(u).
\end{equation}
For $j\ge2$, test \eqref{eq:F-equation} against
\[
  \Psi_j(F):=\abs F^{j-2}F.
\]
The transport term vanishes because
\[
  \int_{\R^3}u\cdot\nabla F\cdot\Psi_j(F)\dd x
  =\frac1j\int_{\R^3}u\cdot\nabla\abs F^j\dd x=0.
\]
The fractional dissipative term is nonnegative.  Indeed, the singular-integral representation of $\Lambdaop^{2\alpha}$ gives
\begin{equation}\label{eq:vector-dissipation-positivity}
\begin{aligned}
  &\int_{\R^3}\Psi_j(F(x))\cdot\Lambdaop^{2\alpha}F(x)\dd x\\
  &\quad=\frac{c_{\alpha}}2
  \iint_{\R^3\times\R^3}
  \frac{(F(x)-F(y))\cdot
  \bigl(\Psi_j(F(x))-\Psi_j(F(y))\bigr)}
  {\abs{x-y}^{3+2\alpha}}\dd x\dd y
  \ge0,
\end{aligned}
\end{equation}
because $z\mapsto\abs z^{j-2}z$ is monotone.

For $1<j<2$, introduce the regularized convex primitive
\begin{equation}\label{eq:regularized-convex-primitive}
\begin{aligned}
  \Phi_{j,\tau}(z)
  &:=\int_0^{\abs z}r(r^2+\tau^2)^{(j-2)/2}\dd r\\
  &=\frac{(\abs z^2+\tau^2)^{j/2}-\tau^j}{j},
  \qquad \tau>0,
\end{aligned}
\end{equation}
and its gradient
\begin{equation}\label{eq:regularized-duality-map}
  \Psi_{j,\tau}(z)
  :=\nabla_z\Phi_{j,\tau}(z)
  =(\abs z^2+\tau^2)^{(j-2)/2}z.
\end{equation}
The tangential eigenvalues of $D\Psi_{j,\tau}(z)$ are $(\abs z^2+\tau^2)^{(j-2)/2}$, while the radial eigenvalue is
\[
  (\abs z^2+\tau^2)^{(j-4)/2}
  \bigl(\tau^2+(j-1)\abs z^2\bigr).
\]
They are positive, so $\Psi_{j,\tau}$ is monotone.  Testing \eqref{eq:F-equation} against $\Psi_{j,\tau}(F)$ gives
\begin{equation}\label{eq:regularized-Lj-energy}
\begin{aligned}
  \frac{\dd}{\dd t}\int_{\R^3}\Phi_{j,\tau}(F)\dd x
  &+\int_{\R^3}\Psi_{j,\tau}(F)\cdot
  \Lambdaop^{2\alpha}F\dd x\\
  &=-\int_{\R^3}\mathcal C_\sigma(u)\cdot
  \Psi_{j,\tau}(F)\dd x.
\end{aligned}
\end{equation}
The transport contribution is $\int u\cdot\nabla\Phi_{j,\tau}(F)\dd x=0$, and the dissipative contribution is nonnegative by the same double-integral formula and the monotonicity of $\Psi_{j,\tau}$.  Since $j-2<0$,
\begin{equation}\label{eq:regularized-dual-bound}
  \abs{\Psi_{j,\tau}(F)}\le\abs F^{j-1}.
\end{equation}
After integrating \eqref{eq:regularized-Lj-energy} over $[0,t]$ and discarding the nonnegative dissipative term, we obtain
\begin{align*}
  \int_{\R^3}\Phi_{j,\tau}(F(t))\dd x
  &\le\int_{\R^3}\Phi_{j,\tau}(F(0))\dd x\\
  &\quad+\int_0^t
  \norm{L^j}{\mathcal C_\sigma(u)(s)}
  \norm{L^j}{F(s)}^{j-1}\dd s.
\end{align*}
As $\tau\downarrow0$, one has
\[
  \Phi_{j,\tau}(z)\uparrow\frac1j\abs z^j.
\]
Monotone convergence therefore yields
\begin{equation}\label{eq:Lj-integrated-inequality}
  \frac1j\norm{L^j}{F(t)}^j
  \le\frac1j\norm{L^j}{F(0)}^j
  +\int_0^t
  \norm{L^j}{\mathcal C_\sigma(u)(s)}
  \norm{L^j}{F(s)}^{j-1}\dd s.
\end{equation}
Applying the standard scalar regularization
$(\norm{L^j}{F}^j+\varepsilon)^{1/j}$ and then letting $\varepsilon\downarrow0$ gives, for almost every $t$,
\begin{equation}\label{eq:Lj-differential-inequality}
  \frac{\dd}{\dd t}\norm{L^j}{F}
  \le\norm{L^j}{\mathcal C_\sigma(u)}.
\end{equation}
For $j\ge2$, the same differential inequality follows directly from \eqref{eq:vector-dissipation-positivity}.  Lemma \ref{lem:projected-commutator} now gives
\[
  \frac{\dd}{\dd t}\norm{L^j}{F}
  \le C_{\sigma,j}\norm{L^\infty}{\nabla u}\norm{L^j}{F}.
\]
Gronwall's inequality proves \eqref{eq:finite-ell-propagation}.

Under \eqref{eq:lipschitz-bootstrap}, one has $Mt_*\le C_\zeta Lt_*=C_\zeta\zeta^{-N}$, which is independent of $\mu$ and $\nu$.  Lemma \ref{lem:initial-smallness} gives the initial $W^{\sigma,\ell}$ scale, and Lemma \ref{lem:approx-profile-bounds} gives the corresponding bound for $\ubar$.  This proves \eqref{eq:finite-ell-scale}.
\end{proof}

\subsection{Endpoint bounds}

\begin{corollary}[Endpoint bounds with anisotropic loss]\label{cor:endpoint-bounds}
Under \eqref{eq:lipschitz-bootstrap}, for every integer $1\le m\le K_{\mathrm b}+1$,
\begin{equation}\label{eq:endpoint-bounds}
  \norm{L^\infty}{\nabla^m u(t)}
  +\norm{L^\infty}{\nabla^m\ubar(t)}
  \lesssim_{\zeta,m}A\mu^m\vartheta^{-\kappa},
  \qquad 0\le t\le T_{\mathrm{boot}}.
\end{equation}
\end{corollary}

\begin{proof}
Sobolev embedding at the fixed finite exponent $\ell$ and Proposition \ref{prop:finite-ell} give
\begin{align*}
  \norm{L^\infty}{\nabla^m u}
  &\lesssim\norm{W^{m+3/\ell+\eta,\ell}}{u}
  \lesssimz S_{m+3/\ell+\eta,\ell}\\
  &=A\mu^m\mu^{\eta+1/\ell}\nu^{-1/\ell}
  =A\mu^m\mu^\eta\vartheta^{-1/\ell}.
\end{align*}
By \eqref{eq:ell-eta-choice},
\[
  \mu^\eta\vartheta^{-1/\ell}
  =\mu^{\eta+b/\ell}
  \le\mu^{b\kappa}
  =\vartheta^{-\kappa}
\]
for $\mu\ge1$.  The same argument applies to $\ubar$, or one may use Lemma \ref{lem:approx-profile-bounds} directly.
\end{proof}

\subsection{Higher-order energy estimates}

Set
\begin{equation}\label{eq:Bk-def}
  B_k:=S_{k,2}=A\mu^{k-1}\nu^{-1/2},
  \qquad 0\le k\le K_{\mathrm b}.
\end{equation}
By Proposition \ref{prop:approximation-package},
\begin{equation}\label{eq:residual-L2-scale}
  \norm{L^2}{\nabla^kE(t)}
  \lesssim_{\zeta,k}B_kA\nu
  =B_kL\vartheta,
  \qquad 0\le k\le K_{\mathrm b}.
\end{equation}

\begin{proposition}[Anisotropic $H^k$ stability]\label{prop:Hk-stability}
For every integer $0\le k\le K_{\mathrm b}$ and every $0\le t\le T_{\mathrm{boot}}$,
\begin{equation}\label{eq:Hk-stability}
  \norm{L^2}{\nabla^kw(t)}
  \le C_{\zeta,k}B_k\vartheta^{\gamma_k}.
\end{equation}
\end{proposition}

\begin{proof}
We proceed by induction on $k$.

\medskip\noindent\textit{Step 1: the $L^2$ estimate.}\quad For $k=0$, take the $L^2$ inner product of \eqref{eq:error-equation} with $w$.  The $u\cdot\nabla w$ term and the pressure term vanish.  Hence
\[
  \frac12\frac{\dd}{\dd t}\norm{L^2}{w}^2
  +\norm{L^2}{\Lambdaop^\alpha w}^2
  \le\norm{L^\infty}{\nabla\ubar}\norm{L^2}{w}^2
  +\norm{L^2}{E}\norm{L^2}{w}.
\]
Using \eqref{eq:approx-lipschitz}, \eqref{eq:residual-L2-scale}, and $w(0)=0$, Gronwall gives
\[
  \norm{L^2}{w(t)}
  \le C_\zeta t_*B_0L\vartheta
  =C_\zeta B_0\vartheta.
\]
Since $\gamma_0=1$, this is \eqref{eq:Hk-stability} at order zero.

\medskip\noindent\textit{Step 2: the higher-order induction.}\quad Assume \eqref{eq:Hk-stability} holds for all orders below an integer $1\le k\le K_{\mathrm b}$.  Let $\beta$ be a multi-index with $\abs\beta=k$.  The high-order transport term expands as
\begin{equation}\label{eq:transport-Leibniz}
  D^\beta(u\cdot\nabla w)
  =u\cdot\nabla D^\beta w
  +\sum_{0<\gamma\le\beta}
  \binom{\beta}{\gamma}D^\gamma u\cdot\nabla D^{\beta-\gamma}w.
\end{equation}
The first term cancels after pairing with $D^\beta w$.  In the sum, the terms with $\abs\gamma=1$ satisfy
\begin{equation}\label{eq:transport-endpoint-term}
  \left|\int D^\gamma u\cdot\nabla D^{\beta-\gamma}w\cdot D^\beta w\dd x\right|
  \lesssim\norm{L^\infty}{\nabla u}\norm{L^2}{\nabla^kw}^2.
\end{equation}
If $\abs\gamma\ge2$, then $\abs{\beta-\gamma}+1\le k-1$.  Putting the derivative of $u$ in $L^\infty$ and the strictly lower derivative of $w$ in $L^2$ gives
\begin{equation}\label{eq:transport-lower-term}
  \left|\int D^\gamma u\cdot\nabla D^{\beta-\gamma}w\cdot D^\beta w\dd x\right|
  \lesssim
  \norm{L^\infty}{\nabla^{\abs\gamma}u}
  \norm{L^2}{\nabla^{k-\abs\gamma+1}w}
  \norm{L^2}{\nabla^kw}.
\end{equation}
Thus no derivative of $w$ above order $k$ is present.

Similarly,
\begin{equation}\label{eq:stretching-Leibniz}
  D^\beta(w\cdot\nabla\ubar)
  =D^\beta w\cdot\nabla\ubar
  +\sum_{\gamma<\beta}
  \binom{\beta}{\gamma}D^\gamma w\cdot\nabla D^{\beta-\gamma}\ubar.
\end{equation}
The first term is bounded by
\begin{equation}\label{eq:stretching-top-term}
  \norm{L^\infty}{\nabla\ubar}\norm{L^2}{\nabla^kw}^2,
\end{equation}
and every term in the sum contains at most $k-1$ derivatives on $w$:
\begin{equation}\label{eq:stretching-lower-term}
  \left|\int D^\gamma w\cdot\nabla D^{\beta-\gamma}\ubar\cdot D^\beta w\dd x\right|
  \lesssim
  \norm{L^2}{\nabla^{\abs\gamma}w}
  \norm{L^\infty}{\nabla^{k-\abs\gamma+1}\ubar}
  \norm{L^2}{\nabla^kw}.
\end{equation}
The pressure again vanishes because $D^\beta w$ is divergence free, while the dissipative term is $\norm{L^2}{\Lambdaop^\alpha D^\beta w}^2\ge0$.

Summing over $\abs\beta=k$, using Corollary \ref{cor:endpoint-bounds}, and dividing by $\norm{L^2}{\nabla^kw}$ through a standard regularization if necessary, we obtain
\begin{equation}\label{eq:Hk-differential}
\begin{aligned}
  \frac{\dd}{\dd t}\norm{L^2}{\nabla^kw}
  &\le C_\zeta L\norm{L^2}{\nabla^kw}\\
  &\quad+C_\zeta\sum_{m=0}^{k-1}
  A\mu^{k+1-m}\vartheta^{-\kappa}
  \norm{L^2}{\nabla^mw}
  +C_\zeta B_kL\vartheta.
\end{aligned}
\end{equation}
By the induction hypothesis and the exact scale identity
\begin{equation}\label{eq:scale-identity-induction}
  A\mu^{k+1-m}B_m=L B_k,
\end{equation}
the lower-order sum in \eqref{eq:Hk-differential} is bounded by
\[
  C_\zeta B_kL\vartheta^{\gamma_{k-1}-\kappa}.
\]
The exponent design gives
\begin{equation}\label{eq:gamma-gain}
  \gamma_{k-1}-\kappa
  =\gamma_k+\kappa>\gamma_k.
\end{equation}
Integrating over $0\le t\le t_*$, using $Lt_*=\zeta^{-N}$ and absorbing the fixed exponential into $C_\zeta$, yields
\[
  \norm{L^2}{\nabla^kw(t)}
  \le C_\zeta B_k
  \left(\vartheta^{\gamma_k+\kappa}+\vartheta\right)
  \le C_\zeta B_k\vartheta^{\gamma_k}.
\]
This closes the induction.
\end{proof}

\begin{corollary}[Real-order stability]\label{cor:real-order-stability}
For every real $0\le\sigma\le K_{\mathrm b}$ and $0\le t\le T_{\mathrm{boot}}$,
\begin{equation}\label{eq:real-order-stability}
  \norm{\dot H^\sigma}{w(t)}
  \lesssim_{\zeta,\sigma}
  A\mu^{\sigma-1}\nu^{-1/2}\vartheta^{\gamma_*}.
\end{equation}
The same estimate with $H^\sigma$ in place of $\dot H^\sigma$ holds after adding the lower-order $L^2$ contribution.
\end{corollary}

\begin{proof}
Interpolate between the two neighboring integer estimates in Proposition \ref{prop:Hk-stability}.  The scale $B_k=A\mu^{k-1}\nu^{-1/2}$ interpolates exactly to $A\mu^{\sigma-1}\nu^{-1/2}$.  Since $\gamma_k\ge\gamma_*$ for every $k\le K_{\mathrm b}$ and $0<\vartheta<1$, the interpolated smallness factor is bounded by $\vartheta^{\gamma_*}$.
\end{proof}

\subsection{Target-space stability and closure}

\begin{lemma}[Closure of the Lipschitz bootstrap]\label{lem:bootstrap-closure}
For each fixed sufficiently small $\zeta$, there exists $\mu_1(\zeta)$ such that, whenever $\mu\ge\mu_1(\zeta)$, the exact solution exists smoothly on $[0,t_*]$ and
\begin{equation}\label{eq:closed-lipschitz}
  \norm{L^\infty}{\nabla u(t)}
  <2C_{\mathrm{app}}(\zeta)L,
  \qquad 0\le t\le t_*.
\end{equation}
All estimates above therefore hold on the entire interval $[0,t_*]$.
\end{lemma}

\begin{proof}
Because $K_{\mathrm b}>5/2+\eta$, Corollary \ref{cor:real-order-stability} and Sobolev embedding imply
\begin{align}
  \norm{L^\infty}{\nabla w(t)}
  &\lesssim\norm{H^{5/2+\eta}}{w(t)}\notag\\
  &\lesssimz A\mu^{3/2+\eta}\nu^{-1/2}\vartheta^{\gamma_*}\notag\\
  &=L\mu^\eta\vartheta^{\gamma_*-1/2}.
  \label{eq:Lipschitz-error-small}
\end{align}
Now
\[
  \mu^\eta\vartheta^{\gamma_*-1/2}
  =\mu^{\eta-b(\gamma_*-1/2)}\longrightarrow0,
\]
because $\gamma_*-1/2>1/4$ and \eqref{eq:ell-eta-choice} gives $\eta<b\kappa/4<b/4$.  Hence, after increasing $\mu_1(\zeta)$,
\[
  \norm{L^\infty}{\nabla w(t)}
  \le\frac12C_{\mathrm{app}}(\zeta)L
  \qquad\text{on }[0,T_{\mathrm{boot}}].
\]
Together with \eqref{eq:approx-lipschitz}, this strictly improves the bootstrap bound.  The continuation criterion in Proposition \ref{prop:local-theory} prevents termination before $t_*$, so $T_{\mathrm{boot}}=t_*$ and \eqref{eq:closed-lipschitz} follows.
\end{proof}

\begin{lemma}[Target-space perturbation]\label{lem:target-perturbation}
At the inflation time $t_*$,
\begin{equation}
  \norm{W^{s,p}}{w(t_*)}
  \lesssimz S_{s,p}\vartheta^{\beta_p},
\end{equation}
and, for every $1\le q\le\infty$,
\begin{equation}
  \norm{B^s_{p,q}}{w(t_*)}
  \lesssim_{\zeta,s,p,q,\eta}
  S_{s,p}\mu^\eta\vartheta^{\beta_p}.
\end{equation}
Both right-hand sides tend to zero relative to $S_{s,p}=\zeta^2$ as $\mu\to\infty$ with $\zeta$ fixed.
\end{lemma}

\begin{proof}
By \eqref{eq:sobolev-transfer-prelim} and Corollary \ref{cor:real-order-stability},
\begin{align*}
  \norm{\dot W^{s,p}}{w(t_*)}
  &\lesssim\norm{\dot H^{s+d_p}}{w(t_*)}\\
  &\lesssimz A\mu^{s+d_p-1}\nu^{-1/2}\vartheta^{\gamma_*}.
\end{align*}
The scales satisfy
\begin{equation}\label{eq:scale-conversion-W}
\begin{aligned}
  A\mu^{s+d_p-1}\nu^{-1/2}
  &=S_{s,p}
  \left(\frac\mu\nu\right)^{1/2-1/p}\\
  &=S_{s,p}\vartheta^{-(1/2-1/p)}.
\end{aligned}
\end{equation}
This proves the homogeneous part of \eqref{eq:W-target-error}.  For the lower-order term, \eqref{eq:sobolev-transfer-prelim} with $s=0$ and Corollary \ref{cor:real-order-stability} give
\begin{align}
  \norm{L^p}{w(t_*)}
  &\lesssim \norm{H^{d_p}}{w(t_*)}\notag\\
  &\lesssimz A\mu^{d_p-1}\nu^{-1/2}\vartheta^{\gamma_*}\notag\\
  &=S_{s,p}\mu^{-s}
  \vartheta^{-(1/2-1/p)}\vartheta^{\gamma_*}\notag\\
  &=S_{s,p}\mu^{-s}\vartheta^{\beta_p}
  \le S_{s,p}\vartheta^{\beta_p},
  \label{eq:lower-order-Lp-transfer}
\end{align}
where the last inequality uses $s>0$ and $\mu\ge1$.  Thus the inhomogeneous $W^{s,p}$ norm satisfies \eqref{eq:W-target-error}.

For the Besov estimate, \eqref{eq:besov-transfer-prelim} gives
\begin{align*}
  \norm{B^s_{p,q}}{w(t_*)}
  &\lesssim\norm{H^{s+d_p+\eta}}{w(t_*)}\\
  &\lesssimz A\mu^{s+d_p+\eta-1}\nu^{-1/2}\vartheta^{\gamma_*}\\
  &=S_{s,p}\mu^\eta\vartheta^{\beta_p}.
\end{align*}
All required real orders are below $K_{\mathrm b}$ by \eqref{eq:Kb-choice}.  Finally,
\[
  \mu^\eta\vartheta^{\beta_p}
  =\mu^{\eta-b\beta_p}\longrightarrow0,
\]
because $\beta_p>1/4$ and $\eta<b\kappa/4<b/4$.
\end{proof}

\begin{proof}[Proof of Proposition \ref{prop:stability}]
Lemma \ref{lem:bootstrap-closure} gives smooth existence on $[0,t_*]$.  Lemma \ref{lem:target-perturbation} then gives \eqref{eq:W-target-error} and \eqref{eq:B-target-error}.  The decay relative to $S_{s,p}$ follows from $\vartheta=\mu^{-b}$ and $\eta<b\beta_p$.
\end{proof}

\section{Proof of the main theorem}\label{sec:main-proof}

Fix
\begin{equation}\label{eq:fixed-target-space}
  X=W^{s,p}
  \qquad\text{or}\qquad
  X=B^s_{p,q}\quad\text{with $q$ fixed}.
\end{equation}
By Proposition \ref{prop:approx-growth}, there are constants $c_{\mathrm{grow}}(X)>0$ and $C_{\mathrm{osc}}(X)\ge1$ such that
\begin{equation}\label{eq:approx-X-growth}
  \norm{X}{\ubar(t_*)}
  \ge c_{\mathrm{grow}}(X)\zeta^{2-Ns}
\end{equation}
whenever $\zeta^{-N}\ge C_{\mathrm{osc}}(X)$.  In the Besov case these constants may depend on the fixed $q$.

Let $C_{\mathrm{in}}(X)$ be the constant in Proposition \ref{prop:approximation-package}, and set
\[
  c_{\mathrm{out}}(X):=\frac12c_{\mathrm{grow}}(X).
\]
Choose $0<\zeta_0(X)<1$ so small that
\begin{equation}\label{eq:zeta0-conditions}
  C_{\mathrm{in}}(X)\zeta_0(X)\le1,
  \qquad
  c_{\mathrm{out}}(X)\zeta_0(X)^{3-Ns}\ge1,
  \qquad
  \zeta_0(X)^{-N}\ge C_{\mathrm{osc}}(X).
\end{equation}
This is possible because $Ns>10$ by \eqref{eq:N-choice}.

\begin{proof}[Proof of Theorem \ref{thm:main}]
Given $\veps>0$, set
\begin{equation}\label{eq:internal-zeta}
  \zeta:=\min\{\veps,\zeta_0(X)\}.
\end{equation}
After fixing $\zeta$, take $\mu$ sufficiently large so that Proposition \ref{prop:approximation-package} and Proposition \ref{prop:stability} hold, $t_*\le\zeta$, and
\begin{equation}\label{eq:final-error-absorption}
  \norm{X}{u(t_*)-\ubar(t_*)}
  \le c_{\mathrm{out}}(X)\zeta^{2-Ns}.
\end{equation}
The last condition is possible because the error in Proposition \ref{prop:stability} is $o_{\mu\to\infty}(\zeta^2)$, while $0<\zeta\le1$ is fixed.

The initial estimate and the first condition in \eqref{eq:zeta0-conditions} give
\[
  \norm{X}{u_0}
  \le C_{\mathrm{in}}(X)\zeta^2
  \le\zeta
  \le\veps.
\]
The exact solution is smooth on $[0,t_*]$ and $0<t_*\le\zeta\le\veps$.  By \eqref{eq:approx-X-growth}, \eqref{eq:final-error-absorption}, and the triangle inequality,
\[
  \norm{X}{u(t_*)}
  \ge c_{\mathrm{out}}(X)\zeta^{2-Ns}
  \ge\zeta^{-1}
  \ge\veps^{-1},
\]
where the second inequality follows from \eqref{eq:zeta0-conditions}.  This proves the theorem for either choice of $X$.
\end{proof}

\end{document}